\documentclass[a4paper,12pt]{amsart}
\usepackage{amssymb,amsmath,amsfonts,amsthm}
\usepackage[all]{xy}
\usepackage{enumerate}
\usepackage{mathrsfs}
\usepackage{graphicx}

\newcommand{\R}{\mathbb{R}}

\newcommand{\N}{\mathbb{N}}
\newcommand{\PP}{\mathbb{P}}

\newcommand{\point}{\mathfrak{p}}

\newcommand{\moddn}{{\operatorname{mod}}}

\numberwithin{equation}{section}
\newtheorem{pro}{Proposition}[section]
\newtheorem{Lem}[pro]{Lemma}
\newtheorem{Cor}[pro]{Corollary}
\newtheorem{Theo}[pro]{Theorem}

\newtheorem*{Theoetoile}{Theorem} 
\theoremstyle{definition}

\newtheorem{Conj}[pro]{Conjecture}
\newtheorem{Rem}[pro]{Remark}

\usepackage{hyperref}
\hypersetup{
	colorlinks=true,       
	linkcolor=red,          
	citecolor=blue,        
	filecolor=magenta,      
	urlcolor=cyan           
}

\title{{\footnotesize On the weak Lefschetz property for almost complete intersections generated by uniform powers of general linear forms}}
\date{\today}

\author{Rosa M. Mir\'o-Roig}
\address{Universitat de Barcelona, Departament de Matem\`atiques i Inform\`atica, Gran Via de les Corts Catalanes 585, 08007 Barcelona, Spain.}
\email{miro@ub.edu}

\author{Quang Hoa Tran}
\address{University of Education, Hue University,  34 Le Loi St., Hue City, Vietnam.}
\email{tranquanghoa@hueuni.edu.vn}

\begin{document}
\maketitle

\begin{abstract}
In \cite{MMN2012}, Conjecture~6.6, Migliore, the first author, and Nagel conjectured that, for all $n\geq 4$, the artinian ideal $I=(L_0^d,\ldots,L_{2n+1}^d) \subset R=k[x_0,\ldots,x_{2n}]$ generated by the $d$-th powers of $2n+2$ general linear forms fails to have the weak Lefschetz property if and only if $d>1$. This paper is entirely devoted to prove partially this conjecture. More precisely,  we prove that $R/I$ fails to have  the weak Lefschetz property, provided $4\leq n\leq 8,\ d\geq 4$ or  $d=2r,\ 1\leq r\leq 8,\ 4\leq n\leq 2r(r+2)-1$.
\texttt{MSC2010}: primary 14C20, 13E10; secondary 13C40, 13C13, 13D02, 13D40.
\end{abstract}

\section{Introduction}
Ideals generated by powers of linear forms have attracted great deal of attention recently. For instance, their Hilbert function has been the focus of the papers \cite{AFA2010,HSS2011,SX2010}; and the presence or failure of the weak Lefschetz property has been deeply studied in \cite{HSS2011,MMN2012, MiroRoig2016,NT2018,SS2010}, among others.

Let $k$ be a field of characteristic zero  and $R=k[x_0,\ldots,x_n]$ be the standard graded polynomial ring over $k$ in $n+1$ variables. A graded artinian $k$-algebra $A:=R/I$ is said to have the  \textit{weak Lefschetz property} (WLP for short) if there is a linear form $\ell\in [A]_1$ such that the multiplication
$$\times \ell: [A]_{i}\longrightarrow [A]_{i+1}$$
has maximal rank for all $i$, i.e., $\times \ell$ is either injective or surjective, for all  $i$. On the contrary, we say that \textit{$A$ fails to have the WLP} if there is an integer $i$ such that the above multiplication does not have maximal rank for any linear form $\ell$. There has been a long series of papers determining classes of algebras holding/failing the WLP but much more work remains to be done.

The first result in this direction is due to Stanley \cite{Stanley1980} and Watanabe \cite{Watanabe1987} and it asserts that the WLP holds for \textit{any} artinian complete intersection ideal $I$ generated by powers of linear forms. In fact, they showed that there is a linear form $\ell\in [A]_1$ such that the multiplication
$$\times \ell^s: [A]_{i}\longrightarrow [A]_{i+s}$$
has maximal rank for all $i,s$. When this property holds, the algebra is said to have the \textit{strong Lefschetz property} (briefly SLP). In \cite{SS2010}, Schenck and Seceleanu gave the nice result that \textit{any} artinian ideal $I\subset R=k[x,y,z]$ generated by powers of linear forms has the WLP. Moreover, when these linear forms are general, the SLP of $R/I$ has also been studied, in particular, the multiplication by the square $\ell^2$ of a general linear form $\ell$ induces a homomorphism of maximal rank in any graded component of $R/I$, see \cite{AA18,MMR17}. However, Migliore, the first author, and Nagel  showed by examples that in 4 variables, an ideal generated by the $d$-th powers of five general linear forms fails to have the WLP for $d =3,\ldots,12$ \cite{MMN2012}. Therefore, it is natural to ask when the WLP holds for artinian ideals $I\subset k[x_0,\ldots,x_n]$ generated by powers of $\geq n+2$ general linear forms. In \cite{MMN2012},  Migliore, the first author, and Nagel studied this question where the ideal is an almost complete intersection and they also  proposed the following conjecture in order to complete this investigation. 
\begin{Conj}\cite[Conjecture~6.6]{MMN2012}\label{Conj1.1}
Let $R=k[x_0,\ldots,x_{2n}]$ be the polynomial ring over a field of characteristic zero. Consider an artinian ideal $I=(L_0^d,\ldots,L_{2n+1}^d)\subset R$ generated by the $d$-th powers of general linear forms. If $n\geq 4$, then the ring $R/I$ fails to have the WLP if and only if $d>1.$  Furthermore, if $n=3$, then $R/I$ fails to have the WLP when $d=3$.
\end{Conj}

The first author has shown that $R/I$ fails to have the WLP when $d=2$ \cite{MiroRoig2016} and in the recent paper \cite{NT2018}, Nagel and Trok have established Conjecture~\ref{Conj1.1} for $n\gg 0$ and $d\gg 0$.  The last part of the conjecture was proved by Di Gennaro, Ilardi, and Vall\`es in \cite[Proposition~5.5]{GIV2014}. Unfortunately, there was a gap in their proof. However, it was corrected in  \cite{GIV19} and then in \cite{GV19}, the last part of Conjecture~\ref{Conj1.1} was proved by Ilardi and Vall\`es. The goal of this note is to solve partially  the conjecture. More precisely, we prove the following (see Corollaries~\ref{Corollary3.3}--\ref{Corollary3.10}, Theorem~\ref{Theorem4.1} and Remark~\ref{Remark4.2}).
\begin{Theoetoile}
Let $R=k[x_0,\ldots,x_{2n}]$ be the polynomial ring over a field of characteristic zero  and consider an artinian ideal $I=(L_0^d,\ldots,L_{2n+1}^d)\subset R$ generated by the $d$-th powers of general linear forms. 
\begin{enumerate}
\item[\rm (1)] If $d=2r,\, 2\leq r\leq 8$ and $4\leq n\leq 2r(r+2)-1,$ then $R/I$ fails to have the WLP.
\item[\rm (2)] If $4\leq n\leq 8$ and $d\geq 4$, then $R/I$ fails to have the WLP.
\end{enumerate}
\end{Theoetoile}
 Therefore, Theorem  answers partially Conjecture~\ref{Conj1.1} for $4\leq n\leq 8$, missing only the case $d=3$, since the case $d=2$ is shown by the first author \cite{MiroRoig2016}.  Our approach is based on the connection between computing the dimension of $R/(I,\ell)$, where $\ell$ is a general linear form and the dimension of linear system of fat points. More precisely, we prove the following result (see Theorem~\ref{Theorem3.1}).
\begin{Theoetoile}
If $\ell$ is a general linear form and $j=\lfloor\frac{(2n^2-1)(d-1)}{2n-1}\rfloor$, then
\begin{align*}
\dim_k&[R/(I,\ell)]_j =\dim_k \mathfrak{L}_{2n-1}\big(j; (j+1-d)^{2n+2}\big)\\
&=\begin{cases}
\dim_k \mathfrak{L}_{2n-1}\big(e; 0^{2n+2}\big) & \text{if}\; d=(2n-1)e+1\\
\dim_k \mathfrak{L}_{2n-1}\big(e+n-r+1; (n-r)^{2n+2}\big) & \text{if}\; d=(2n-1)e+2r\\
\dim_k \mathfrak{L}_{2n-1}\big(e+2n-r+1; (2n-r-1)^{2n+2}\big) & \text{if}\; d=(2n-1)e+2r+1
\end{cases}
\end{align*}
 where $e,r$ are non-negative integers such that $ 1\leq r\leq n-1$.
 \end{Theoetoile}
\section{Preparatory results}
Throughout this paper $R$ denotes a polynomial ring $k[x_0,\ldots,x_n]$ over a field $k$ of characteristic zero, with its standard grading where $\deg(x_i)=1.$ If $I\subset R$ is a homogeneous ideal, then the $k$-algebra $A=\oplus_{j\geq 0} [A]_j$ is standard graded. Its Hilbert function is a map $h_A: \N\longrightarrow \N, h_A(j)=\dim_k [A]_j.$

For any artinian ideal $I\subset R$  and a general linear form $\ell\in R$, the exact sequence
$$ [R/I]_{j-1}\longrightarrow [R/I]_j\longrightarrow [R/(I,\ell)]_j\longrightarrow 0$$
gives, in particular, that the multiplication by $\ell$  will fail to have maximal rank exactly when
\begin{equation}\label{equation2.1}
\dim_k[R/(I,\ell)]_j\neq \max\{ \dim_k[R/I]_j -\dim_k[R/I]_{j-1}, 0 \}.
\end{equation}
In this case, we will say that $R/I$ fails to have the WLP in degree $j-1.$

We first recall a result of Emsalem and Iarrobino giving a duality between powers of linear forms and ideals of fat points in $\PP^n$. We quote it in the form that we need.
\begin{Lem}\cite[Theorem I]{EI1995}\label{Lemma2.1}
Let $\point_1, \ldots, \point_s$ be the ideals of $s$ distinct points in $\PP^n$ that are dual to the linear forms $\ell_1,\ldots,\ell_s\in R$ and choose the positive integers $a_1,\ldots,a_s$. Then, for each integer $j\geq -1+\max\{a_1,\ldots,a_s\}$,
$$\dim_k\big[R/(\ell_1^{a_1},\ldots,\ell_s^{a_s})\big]_j =\dim_k\big[\bigcap_{a_i\leq j} \point_i^{j+1-a_i}\big]_j.$$
\end{Lem}
If the points defined by the ideals $\point_i$ are general points, then the dimension of the linear system $[\point_1^{b_1}\cap \cdots\cap \point_s^{b_s}]_j\subset R_j$ depends only on the numbers $n, j, b_1,\ldots,b_s.$ In order to simplify notation, in this case we denote by
$$\mathfrak{L}_n(j; b_1,\ldots,b_s)$$
the linear system $[\point_1^{b_1}\cap \cdots\cap \point_s^{b_s}]_j\subset R_j$. We use superscripts to indicate repeated entries. For example, $\mathfrak{L}_4(j; 2^4,4^2)=\mathfrak{L}_4(j; 2,2,2,2,4,4).$
Notice that, for every linear system $\mathfrak{L}_n(j; b_1,\ldots,b_s)$, one has
$$\dim_k \mathfrak{L}_n(j; b_1,\ldots,b_s)\geq \max\Big\{0, \binom{n+j}{n} - \sum_{i=1}^{s}\binom{n+b_i-1}{n}\Big\}.$$

To study the WLP,  the following is useful to compute the left-hand side of \eqref{equation2.1}. 
\begin{Lem}\cite[Proposition~3.4]{MMN2012}\label{Lemma2.1.1}
Let $(\ell_1^{a_1},\ldots,\ell_s^{a_s})$ be an ideal of $R$ generated by powers of $s$ general linear forms, and let $\ell$ be a general linear form. Then, for each integer $j\geq -1+\max\{a_1,\ldots,a_s\}$, 
$$\dim_k \big[R/(\ell_1^{a_1},\ldots,\ell_s^{a_s},\ell)\big]_j = \dim_k \mathfrak{L}_{n-1}(j; j+1- a_1,\ldots,j+1-a_s).$$
\end{Lem}

Using Cremona transformations, one can relate two different linear systems. This is often stated only for general points.
\begin{Lem}\cite[Theorem~3]{Dumnicki2009}\label{Lemma2.2}
Let $s>n\geq 2$ and $j, b_1,\ldots,b_s$ be non-negative integers, with $b_1\geq \cdots \geq b_s.$ Set $t=(n-1)j-(b_1+\cdots+b_{n+1})$. If $b_i + t\geq 0$ for all $i=1,\ldots,n+1,$ then
$$\dim_k\mathfrak{L}_n(j; b_1,\ldots,b_s)=\dim_k\mathfrak{L}_n(j+t; b_1+t,\ldots,b_{n+1}+t, b_{n+2},\ldots,b_s).$$
\end{Lem}

In this note, we are interested in certain almost complete intersections. Then one can compute the right-hand side of \eqref{equation2.1}. For any integer $m,$ we denote
 $$[m]_+ =\max\{m,0\}.$$
\begin{Lem}\cite[Lemma~3.7]{MMN2012} \label{Lemma2.3}
Let $I=(L_0^{a_0},\ldots, L_{n+1}^{a_{n+1}})\subset R$ be an almost complete intersection generated by powers of $n+2$ general linear forms. Set $A= R/(L_0^{a_0},\ldots L_{n}^{a_n})$.  Then, for each integer~$j$,
$$\dim_k[R/I]_j -\dim_k[R/I]_{j-1}=[h_A(j)-h_A(j-a_{n+1})]_+ - [h_A(j-1)-h_A(j-a_{n+1}-1)]_+.$$
Furthermore, if $j\leq \frac{1}{2}a_{n+1}+\frac{1}{2}\sum_{i=0}^{n}(a_i-1)$, then the formula simplifies to
$$\dim_k[R/I]_j -\dim_k[R/I]_{j-1}=[h_A(j)-h_A(j-1)]- [h_A(j-a_{n+1})-h_A(j-a_{n+1}-1)].$$
\end{Lem}

\section{Almost uniform powers of general linear forms}
Throughout this section, we always denote $R=k[x_0,\ldots,x_{2n}]$ and consider an artinian ideal $I=(L_0^d,\ldots,L_{2n+1}^d)$  of $R$ generated by the $d$-th powers of general linear forms and fix $j=\lfloor\frac{(2n^2-1)(d-1)}{2n-1}\rfloor$.
\begin{Theo}\label{Theorem3.1}
If $\ell$ is a general linear form, then
\begin{align*}
\dim_k&[R/(I,\ell)]_j =\dim_k \mathfrak{L}_{2n-1}\big(j; (j+1-d)^{2n+2}\big)\\
&=\begin{cases}
\dim_k \mathfrak{L}_{2n-1}\big(e; 0^{2n+2}\big) & \text{if}\; d=(2n-1)e+1\\
\dim_k \mathfrak{L}_{2n-1}\big(e+n-r+1; (n-r)^{2n+2}\big) & \text{if}\; d=(2n-1)e+2r\\
\dim_k \mathfrak{L}_{2n-1}\big(e+2n-r+1; (2n-r-1)^{2n+2}\big) & \text{if}\; d=(2n-1)e+2r+1
\end{cases}
\end{align*}
where $e,r$ are non-negative integers such that $ 1\leq r\leq n-1$  and
\begin{align*}
\dim_k[R/I]_j- \dim_k[R/I]_{j-1}=\sum_{k=0}^{n}(-1)^k\binom{2n+2}{k}\binom{2n-1+j-kd}{2n-1}.
\end{align*}
\end{Theo}
\begin{proof}
It follows from Lemma~\ref{Lemma2.1.1} that
\begin{align*}
D:=&\dim_k[R/(I,\ell)]_j =\dim_k \mathfrak{L}_{2n-1}\big(j; (j+1-d)^{2n+2}\big).
\end{align*}
Set 
$$t=(2n-2)j-2n(j+1-d)=-2j+2n(d-1).$$ 
As $j=\lfloor \frac{(2n^2-1)(d-1)}{2n-1}\rfloor,$ we get 
$$j+1-d+t=-j+(2n-1)(d-1)\geq \frac{2(n-1)^2(d-1)}{2n-1}\geq 0.$$
Using Lemma~\ref{Lemma2.2} $(n+1)$ times, in each step the Cremona transformation changes the  multiplicities of linear system $\mathfrak{L}_{2n-1}\big(j; (j+1-d)^{2n+2}\big)$ by $t$, we obtain
\begin{align}
D=&\dim_k \mathfrak{L}_{2n-1}\big(j; (j+1-d)^{2n+2}\big)\nonumber \\
=&\dim_k \mathfrak{L}_{2n-1}\big(-j+2n(d-1); (j+1-d)^{2}, (-j+(2n-1)(d-1))^{2n}\big)\nonumber \\
=& \dim_k \mathfrak{L}_{2n-1}\big(-3j+4n(d-1); (-j+(2n-1)(d-1))^{4}, (-3j+(4n-1)(d-1))^{2n-2}\big) \nonumber \\
&\vdots \nonumber \\
=& \dim_k \mathfrak{L}_{2n-1}\big(-(2n+1)j+2n(n+1)(d-1); (-(2n-1)j+(2n^2-1)(d-1))^{2n+2}\big)\nonumber .
\end{align}
These computations are correct and has a chance of resulting in a non-empty linear system if
$$-(2n+1)j+2n(n+1)(d-1)> -(2n-1)j+(2n^2-1)(d-1)\geq 0, $$
which is true since $j=\lfloor \frac{(2n^2-1)(d-1)}{2n-1}\rfloor.$ Thus
{ \begin{multline}\label{equation3.1}
D=\dim_k \mathfrak{L}_{2n-1}\big(2n(n+1)(d-1)-(2n+1)j; ((2n^2-1)(d-1)-(2n-1)j)^{2n+2}\big).
\end{multline}}
 Now we consider three cases:
 
\noindent {\bf Case 1}: $d=(2n-1)e+1,$ hence $j=(2n^2-1)e$.  By \eqref{equation3.1} and a simply computation shows that
$$D=\dim_k \mathfrak{L}_{2n-1}\big(e; 0^{2n+2}\big).$$

\noindent{\bf Case 2}: $d=(2n-1)e+2r,\ 1\leq r\leq n-1$. A straightforward computation shows that 
$$j=(2n^2-1)e+ 2nr+r-n-1.$$
Therefore, by \eqref{equation3.1}, we obtain
$$D=\dim_k \mathfrak{L}_{2n-1}\big(e+n-r+1; (n-r)^{2n+2}\big).$$

\noindent {\bf Case 3}: $d=(2n-1)e+2r+1,\ 1\leq r\leq n-1$. It is easy to show that 
$$j=(2n^2-1)e+ 2nr+r-1.$$
By \eqref{equation3.1} we get that
$$D=\dim_k \mathfrak{L}_{2n-1}\big(e+2n-r+1; (2n-r-1)^{2n+2}\big).$$

Finally, let $A=R/(L_0^d,\ldots,L_{2n}^d)$, hence $A$ is a complete intersection and it has the SLP (see, e.g., \cite{Stanley1980} or \cite{Watanabe1987}), one has
$$\dim_k[R/I]_j -\dim_k[R/I]_{j-1}=[h_A(j)-h_A(j-1)]- [h_A(j-d)-h_A(j-d-1)],$$
by  Lemma~\ref{Lemma2.3} since $j\leq \frac{d}{2}+ \frac{(2n+1)(d-1)}{2}.$ 
Resolving $A$ over $R$ using the Koszul resolution, we get for the Hilbert function of~$A$
$$h_A(j)=\sum_{k=0}^{2n+1}(-1)^k \binom{2n+1}{k}\binom{2n+j-kd}{2n}.$$
As $j=\lfloor \frac{(2n^2-1)(d-1)}{2n-1}\rfloor$, hence $j-kd<0$ if $k\geq n+1.$ It follows that
$$h_A(j)= \sum_{k=0}^{n}(-1)^k \binom{2n+1}{k}\binom{2n+j-kd}{2n}.$$
A straightforward computation gives
\begin{align*}
\dim_k[R/I]_j -\dim_k[R/I]_{j-1} = \sum_{k=0}^{n}(-1)^k \binom{2n+2}{k}\binom{2n-1+j-kd}{2n-1}.
\end{align*}
\end{proof}

\begin{pro}\label{Corollary3.2}
Assume that $d=(2n-1)e+2r$, $e$ and $r$ are non-negative integers such that $1\leq r \leq n$. If $n \leq 2r(r+2)-1$ then 
$$\dim_k[R/(I,\ell)]_j >0,$$
where $\ell$ is a general linear form in $R.$
\end{pro}
\begin{proof}
As $d = (2n-1)e+2r,\,  1\leq r\leq n$, by Theorem~\ref{Theorem3.1} we get that 
$$\dim_k[R/(I,\ell)]_j = \dim_k \mathfrak{L}_{2n-1}\big(e+n-r+1; (n-r)^{2n+2}\big),$$
where $e$ is a non-negative integer. It is enough to show that
$$\dim_k \mathfrak{L}_{2n-1}\big(n-r+1; (n-r)^{2n+2}\big)>0.$$
Lemma~\ref{Lemma2.1} shows that
$$\dim_k \mathfrak{L}_{2n-1}\big(n-r+1; (n-r)^{2n+2}\big)=\dim_k\Big[k[x_0,\ldots,x_{2n-1}]/(\ell_0^2,\ldots,\ell_{2n+1}^2)\Big]_{n-r+1},$$
where $\ell_0,\ldots,\ell_{2n+1}$ are general linear forms in $k[x_0,\ldots,x_{2n-1}].$ Setting 
$$P=k[x_0,\ldots,x_{2n-1}]/(\ell_0^2,\ldots,\ell_{2n}^2)$$
then, by \cite[Proposition 3.4]{MiroRoig2016}, for every $0\leq t\leq n$,
$$\dim_k [P]_t = \binom{2n}{t} - \binom{2n}{t-2}.$$
It follows that
\begin{align*}
\dim_k \mathfrak{L}_{2n-1}\big(n-r+1; (n-r)^{2n+2}\big) &\geq \dim_k [P]_{n-r+1} - \dim_k [P]_{n-r-1} \\
&= \binom{2n}{n-r+1} - 2\binom{2n}{n-r -1}+\binom{2n}{n-r-3}.
\end{align*}
We have
\begin{align*}
\binom{2n}{n-r+1}& - 2\binom{2n}{n-r -1}+\binom{2n}{n-r-3}>0\\
&\Leftrightarrow\frac{(2n)!}{(n-r+1)! (n+r-1)!} - \frac{(2n)!}{(n-r-1)! (n+r+1)!}\\
& >  \frac{(2n)!}{(n-r-1)! (n+r+1)!} - \frac{(2n)!}{(n-r-3)! (n+r+3)!} \\
& \Leftrightarrow \frac{2r(2n+1)}{(n-r+1)!(n+r+1)!}> \frac{2(r+2)(2n+1)}{(n-r-1)!(n+r+3)!}\\
&\Leftrightarrow \frac{r}{(n-r)(n-r+1)}> \frac{r+2}{(n+r+2)(n+r+3)}\\
& \Leftrightarrow n^2-(2r^2+4r-1)n-(2r^2+4r)<0\\
& \Leftrightarrow -1< n< 2r(r+2).
\end{align*}
Thus $\dim_k[R/(I,\ell)]_{j}$ for any  $r\leq n\leq 2r(r+2)-1.$
\end{proof}

\begin{Cor} \label{Corollary3.3}
If $2\leq n\leq 15$ and $d=4$, then $R/I$ fails to have the WLP.
\end{Cor}
\begin{proof}
In this case, we have  $j=3n+1$.  By Proposition~\ref{Corollary3.2}, for any $2\leq n\leq 15,$ 
 $$\dim_k[R/(I,\ell)]_{3n+1} >0,$$
where  $\ell$ is a general linear form in $R$. 

On other hand, by Theorem~\ref{Theorem3.1}, we have
\begin{align*}
\dim_k[R/I]_{3n+1} -\dim_k[R/I]_{3n} = \sum_{k=0}^{n}(-1)^k \binom{2n+2}{k}\binom{5n-4k}{2n-1}.
\end{align*}
Using  {\tt Macaulay2} \cite{Macaulay2}, we see that for any $2\leq n\leq 15,$
$$\dim_k[R/I]_{3n+1} -\dim_k[R/I]_{3n} \leq 0.$$
It follows that $R/I$ fails to have the WLP since the surjectivity does not hold.
\end{proof}
\begin{Rem}
Set 
$$S_n=\sum_{k=0}^{n}(-1)^k \binom{2n+2}{k}\binom{5n-4k}{2n-1},n\geq 2.$$
Examples suggest that the sequence $(S_n)_{n\geq 2}$ of integers is strictly decreasing with $S_2=0,$ and so all these are non-positive.
\end{Rem}

\begin{Cor}
If $3\leq n\leq 29$ and $d=6$, then $R/I$ fails to have the WLP.
\end{Cor}
\begin{proof}
In this case, we have $j=5n+2$.  Let $\ell$ be a general linear form in $R$. By Proposition~\ref{Corollary3.2} we get that 
 $$\dim_k[R/(I,\ell)]_{5n+2} >0,$$
for any $3\leq n\leq 29.$  

On other hand, by Theorem~\ref{Theorem3.1}, we have
\begin{align*}
\dim_k[R/I]_{5n+2} -\dim_k[R/I]_{5n+1} = \sum_{k=0}^{n}(-1)^k \binom{2n+2}{k}\binom{7n+1-6k}{2n-1}.
\end{align*}
Using  {\tt Macaulay2} \cite{Macaulay2},  we see that for any $3\leq n\leq 29,$
$$\dim_k[R/I]_{5n+2} -\dim_k[R/I]_{5n+1} < 0,$$
which shows that $R/I$ fails to have the WLP since the surjectivity does not hold.
\end{proof}

\begin{Cor}
If $4\leq n\leq 47$ and $d=8$, then $R/I$ fails to have the WLP.
\end{Cor}
\begin{proof}
Let $\ell$ be a general linear form. In this case, we have $j=7n+3$.  By Proposition~\ref{Corollary3.2}, for any $4\leq n\leq 47,$  
$$\dim_k[R/(I,\ell)]_{7n+3} >0.$$
	
On other hand, by Theorem~\ref{Theorem3.1}, we have
\begin{align*}
\dim_k[R/I]_{7n+3} -\dim_k[R/I]_{7n+2} = \sum_{k=0}^{n}(-1)^k \binom{2n+2}{k}\binom{9n+2-8k}{2n-1}.
\end{align*}
Using  {\tt Macaulay2} \cite{Macaulay2}, we see that for any $4\leq n\leq 47,$
$$\dim_k[R/I]_{7n+3} -\dim_k[R/I]_{7n+2} < 0,$$
which shows that $R/I$ fails to have the WLP since the surjectivity does not hold.
\end{proof}

\begin{Cor}
If $5\leq n\leq 69$ and $d=10$, then $R/I$ fails to have the WLP.
\end{Cor}
\begin{proof}
Let $\ell$ be a general linear form in $R$. In this case, one has $j=9n+4$.  By Proposition~\ref{Corollary3.2}, for any $5\leq n\leq 69,$  
$$\dim_k[R/(I,\ell)]_{9n+4} >0.$$
	
On other hand, by Theorem~\ref{Theorem3.1}, we have
\begin{align*}
\dim_k[R/I]_{9n+4} -\dim_k[R/I]_{9n+3} = \sum_{k=0}^{n}(-1)^k \binom{2n+2}{k}\binom{11n+3-10k}{2n-1}.
\end{align*}
Using  {\tt Macaulay2} \cite{Macaulay2}, we see that for any $5\leq n\leq 69,$
$$\dim_k[R/I]_{9n+4} -\dim_k[R/I]_{9n+3} < 0,$$
which shows that $R/I$ fails to have the WLP since the surjectivity does not hold.
\end{proof}

\begin{Cor}
If $6\leq n\leq 95$ and $d=12$, then $R/I$ fails to have the WLP.
\end{Cor}
\begin{proof}
In this case, one has $j=11n+5$.  For any $6\leq n\leq 95,$ one has
$$\dim_k[R/(I,\ell)]_{11n+5} >0,$$
by Proposition~\ref{Corollary3.2}, where $\ell$ is a general linear form in $R$.
	
On other hand, by Theorem~\ref{Theorem3.1}, we have
\begin{align*}
\dim_k[R/I]_{11n+5} -\dim_k[R/I]_{11n+4} = \sum_{k=0}^{n}(-1)^k \binom{2n+2}{k}\binom{13n+4-12k}{2n-1}.
\end{align*}
Using  {\tt Macaulay2} \cite{Macaulay2}, we see that for any $6\leq n\leq 95,$
$$\dim_k[R/I]_{11n+5} -\dim_k[R/I]_{11n+4} < 0,$$
which shows that $R/I$ fails to have the WLP since the surjectivity does not hold.
\end{proof}

\begin{Cor}
If $7\leq n\leq 125$ and $d=14$, then $R/I$ fails to have the WLP.
\end{Cor}
\begin{proof}
Let $\ell$ be a general linear form in $R$. In this case, one has $j=13n+6$.  Proposition~\ref{Corollary3.2} follows that 
$$\dim_k[R/(I,\ell)]_{13n+6} >0,$$
for any $7\leq n\leq 125.$ 
	
On other hand, by Theorem~\ref{Theorem3.1}, we have
\begin{align*}
\dim_k[R/I]_{13n+6} -\dim_k[R/I]_{13n+5} = \sum_{k=0}^{n}(-1)^k \binom{2n+2}{k}\binom{15n+5-14k}{2n-1}.
\end{align*}
Using  {\tt Macaulay2} \cite{Macaulay2}, we see that for any $7\leq n\leq 125,$
$$\dim_k[R/I]_{13n+6} -\dim_k[R/I]_{13n+5} < 0,$$
which shows that $R/I$ fails to have the WLP since the surjectivity does not hold.
\end{proof}

\begin{Cor}\label{Corollary3.10}
If $8\leq n\leq 159$ and $d=16$, then $R/I$ fails to have the WLP.
\end{Cor}
\begin{proof}
Let $\ell$ be a general linear form in $R$. Computation shows that $j=15n+7$.  Proposition~\ref{Corollary3.2} follows that 
$$\dim_k[R/(I,\ell)]_{15n+7} >0,$$
for any $8\leq n\leq 159.$ 

On other hand, by Theorem~\ref{Theorem3.1}, we have
\begin{align*}
\dim_k[R/I]_{15n+7} -\dim_k[R/I]_{15n+6} = \sum_{k=0}^{n}(-1)^k \binom{2n+2}{k}\binom{17n+6-16k}{2n-1}.
\end{align*}
Using  {\tt Macaulay2} \cite{Macaulay2}, we see that for any $7\leq n\leq 159,$
$$\dim_k[R/I]_{15n+7} -\dim_k[R/I]_{15n+6} < 0,$$
which shows that $R/I$ fails to have the WLP since the surjectivity does not hold.
\end{proof}

\begin{pro} \label{Proposition3.10}
Assume that $n,d\geq 2$ and $\ell$ is a general linear form in $R.$ Then
$$\dim_k[R/(I,\ell)]_j>0 $$
if one of the following conditions is satisfied
\begin{enumerate}
\item [\rm (i)] $2n-1$ or $2n+1$ divides $d-1$.
\item [\rm (ii)] $2n-1$ divides $d+1$.
\item [\rm (iii)] $2n-1$ divides $d+3$.
\item [\rm (iv)] $2n-1$ divides $d+5$.
\item [\rm (v)] $d\geq 4n^2-2n+2.$
\end{enumerate}
\end{pro}
\begin{proof}
Set $t=\lfloor\frac{2n(n+1)(d-1)}{2n+1} \rfloor.$ It is easy to show that $j\leq t$. It follows from \cite[Proposition~4.1]{NT2018} that
$$\dim_k[R/(I,\ell)]_t>0 $$ 
if $2n+1$ divides $d-1$ or $d\geq  4n^2-2n+2.$ Hence 
$$\dim_k[R/(I,\ell)]_j>0 $$ if $2n+1$ divides $d-1$ or $d\geq  4n^2-2n+2$ as claimed in the item (v) and the last part of the item (i). Now if $2n-1$ divides $d-1$, then, by Theorem~\ref{Theorem3.1},
$$\dim_k[R/(I,\ell)]_j=\dim_k \mathfrak{L}_{2n-1}\big(e; 0^{2n+2}\big)>0,\; \forall e\geq 1,$$
which complete the proof of the item (i).

If $d+1=(2n-1)e, \, e\geq 1$, then $d=(2n-1)(e-1)+2(n-1).$ By Theorem~\ref{Theorem3.1},
$$\dim_k[R/(I,\ell)]_j=\dim_k \mathfrak{L}_{2n-1}\big(e+1; 1^{2n+2}\big)>0,\; \forall e\geq 1, n\geq 2$$
as claimed in the item (ii).

If $d+3=(2n-1)e, \, e\geq 1$, then $d=(2n-1)(e-1)+2(n-2).$ If $n\geq 3,$ by Theorem~\ref{Theorem3.1},
$$\dim_k[R/(I,\ell)]_j=\dim_k \mathfrak{L}_{2n-1}\big(e+2; 2^{2n+2}\big)\geq \dim_k \mathfrak{L}_{2n-1}\big(3; 2^{2n+2}\big).$$
As
\begin{align*}
\dim_k \mathfrak{L}_{2n-1}\big(3; 2^{2n+2}\big) &\geq \binom{2n+2}{2n-1} -(2n+2)\binom{2n}{2n-1}\\
& =\frac{2n(n+1)(2n-5)}{3}>0.
\end{align*}
If $n=2$, then $d+3=3e, \, e\geq 2$. It follows that
$$\dim_k[R/(I,\ell)]_j =\dim_k \mathfrak{L}_{3}\big(e+2; 2^{6}\big)>0.$$
Thus (iii) is proved. 

It remains to show (iv). Since $d+5=(2n-1)e$ with $ e\geq 1$ if $n\geq 4$, one has 
$$d=(2n-1)(e-1)+2(n-3).$$
If $n\geq 5,$ by Theorem~\ref{Theorem3.1},
\begin{align*}
\dim_k[R/(I,\ell)]_j&=\dim_k \mathfrak{L}_{2n-1}\big(e+3; 3^{2n+2}\big)\\
&\geq \dim_k \mathfrak{L}_{2n-1}\big(4; 3^{2n+2}\big)\\
&\geq \binom{2n+3}{2n-1} -(2n+2)\binom{2n+1}{2n-1}\\
& =\frac{n(n+1)(2n+1)(2n-9)}{6}>0.
\end{align*}
Note that
$e\geq \begin{cases} 3 &\text{if}\; n=2\\
2&\text{if}\; n=3
\end{cases}.$
Thus
\begin{align*}
\dim_k[R/(I,\ell)]_j =\begin{cases} \dim_k \mathfrak{L}_{3}\big(e-2; 0^{6}\big) &\text{if}\; n=2,e\geq 3\\
\dim_k \mathfrak{L}_{5}\big(e+3; 3^{8}\big) &\text{if}\; n=3,e\geq 2\\
\dim_k \mathfrak{L}_{7}\big(e+3; 3^{10}\big) &\text{if} \; n=4, e\geq 1.
\end{cases}
\end{align*}
Therefore, $\dim_k[R/(I,\ell)]_j >0$ if $n\in \{2,3\}.$ If $n=4$, then 
$$\dim_k[R/(I,\ell)]_j =\dim_k \mathfrak{L}_{7}\big(e+3; 3^{10}\big)\geq \dim_k \mathfrak{L}_{7}\big(4; 3^{10}\big).$$
Set $ P=k[x_0,\ldots,x_7]/(x_0^2,\ldots,x_7^2).$ We have
\begin{align*}
\dim_k\mathfrak{L}_{7}\big(4; 3^{10}\big) \geq h_P(4)-2h_P(2)+h_P(0) = 15>0.
\end{align*}
This completes the argument.
\end{proof}
We close this section by giving the following result. It is similar to a result of Nagel and Trok in \cite[Proposition~6.3]{NT2018}.

\begin{pro} \label{Theorem3.12}
Given integers $n\geq 2$ and $0\leq q\leq 2(n-1),$ define a polynomial function
$P_{n,q}:~\R\longrightarrow \R$ by
{\small $$P_{n,q}(t)=\sum_{k=0}^{n}(-1)^k\binom{2n+2}{k}\binom{n-1+\lfloor\frac{(n-1)q}{2n-1}\rfloor +(q+1)(n-k)+t[2n^2-1-(2n-1)k]}{2n-1}.$$}
Then one has:
\begin{enumerate}
\item [\rm(a)] If for some $q$ with $1\leq q\leq 2(n-1),\ P_{n,q}(t)\leq 0$ for every $t\geq 0$, then Conjecture~\ref{Conj1.1} is true for every $d\geq 4n^2-2n+2$ such that $d-1-q$ is divisible by $2n-1.$
\item [\rm(b)] If $P_{n,0}(t)\leq 0$ for every $t\geq 1$, then Conjecture~\ref{Conj1.1} is true for every $d$ such that $d-1$ is divisible by $2n-1.$
\item [\rm(c)] If $\sum_{k=0}^{n}(-1)^k \binom{2n+2}{k} [2n^2-1-(2n-1)k]^{2n-1}< 0$ for each integer $n\geq 4$, then Conjecture~\ref{Conj1.1} is true for every $d\gg 0$.
\end{enumerate}
\end{pro}
\begin{proof}
Let $\ell$ be a general linear form in $R$. It follows from Proposition~\ref{Proposition3.10} that 
$$\dim_k[R/(I,\ell)]_j>0,$$
provided $d\geq 4n^2-2n+2$ or $d+i$ is divisible by $2n-1$ with $i\in \{-1,1,3,5\}$. It follows that under these assumptions the multiplication
$$\times \ell: [R/I]_{j-1} \longrightarrow [R/I]_j$$
fails to have maximal rank if and only if it fails surjectivity. It is enough to show that
\begin{align*}
\dim_k[R/I]_j- \dim_k[R/I]_{j-1}\leq 0.
\end{align*}
Now write $d-1=(2n-1)t+q$ with integers $t$ and $q$ where $0\leq q\leq 2(n-1).$ Then a straightforward computation gives
$$j=(2n^2-1)t+nq+\big\lfloor\frac{(n-1)q}{2n-1}\big\rfloor.$$
By Theorem~\ref{Theorem3.1},
\begin{align*}
&\dim_k[R/I]_j- \dim_k[R/I]_{j-1}=\sum_{k=0}^{n}(-1)^k\binom{2n+2}{k}\binom{2n-1+j-kd}{2n-1}\\
&=\sum_{k=0}^{n}(-1)^k\binom{2n+2}{k}\binom{n-1+\big\lfloor\frac{(n-1)q}{2n-1}\big\rfloor +(q+1)(n-k)+t[2n^2-1-(2n-1)k]}{2n-1}\\
&=P_{n,q}(t).
\end{align*}
Now, if for some integer $t\geq 0$ we have $P_{n,q}(t)\leq0,$ then 
$$\dim_k[R/(I,\ell)]_j\neq \max\{\dim_k[R/I]_j- \dim_k[R/I]_{j-1},0\}.$$
This proves assertions (a) and (b).

Note that $P_{n,q}(t)$ is a polynomial in $t$ of degree $2n-1$ and  
$$c_n:=\sum_{k=0}^{n}(-1)^k \binom{2n+2}{k} [2n^2-1-(2n-1)k]^{2n-1}$$
is the coefficient of $t^{2n-1}$ in $P_{n,q}(t).$ Since $c_n<0$ by assumption, it follows that $P_{n,q}(t)<0$ for all $t\gg 0$ independent of $q$, and thus the claim (c) is proved.
\end{proof}
Based on computations, we conjecture that
$$c_n:=\sum_{k=0}^{n}(-1)^k \binom{2n+2}{k} [2n^2-1-(2n-1)k]^{2n-1}<0,\; \text{for any}\; n\geq 2.$$
In facts that computations suggest that the sequence $(c_n)_{n\geq 2}$ of integers is strictly decreasing with $c_2=-26,$ and so all these are negatives.
Thank to {\tt Macaulay2} \cite{Macaulay2}, we can check it $c_n<0$ for any $2\leq n\leq 400.$ This conjecture implies that Conjecture~\ref{Conj1.1} is true for $d\gg 0.$

\section{Almost uniform powers of general linear forms in a few variables}
Our main result of this section is the following.
\begin{Theo}\label{Theorem4.1}
Let $R=k[x_0,\ldots,x_{2n}]$ and $I=(L_0^d,\ldots,L_{2n+1}^d)$, where $L_0,\ldots,L_{2n+1}$ are general linear forms in $R$. If $4\leq n\leq 8$ and $d\geq 4$, then $R/I$ fails to have the WLP.
\end{Theo}
\begin{proof}
Let $\ell\in R$ be a general linear form and we will show that the multiplication 
$$\times \ell: [R/I]_{j-1} \longrightarrow [R/I]_j$$
fails to have maximal rank with $j=\lfloor \frac{(2n^2-1)(d-1)}{2n-1}\rfloor,$ provided $4\leq n\leq 8$ and $d\geq 4.$ To do this, we will show
$$\dim_k[R/(I,\ell)]_j \neq \max\{ \dim_k[R/I]_j -\dim_k[R/I]_{j-1}, 0 \}.$$

First, we prove the following assertion.

\noindent \textsc{ Claim 1}: $D:=\dim_k[R/(I,\ell)]_j >0$ for any $4\leq n\leq 8$ and $d\geq 4.$ 

 Indeed, Theorem~\ref{Theorem3.1} shows that
\begin{align*}
D=\begin{cases}
\dim_k \mathfrak{L}_{2n-1}(e; 0^{2n+2}) & \text{if}\quad d=(2n-1)e+1\\
\dim_k \mathfrak{L}_{2n-1}(e+n-r+1; (n-r)^{2n+2}) & \text{if}\quad d=(2n-1)e +2r\\
\dim_k \mathfrak{L}_{2n-1}(e+2n-r+1; (2n-r-1)^{2n+2}) & \text{if}\quad d=(2n-1)e +2r+1
\end{cases}
\end{align*}
where $e,r$ are non-negative integers and $1\leq r\leq n-1$. Note that the dimension of linear systems satisfies
\begin{align}\label{equation4.1}
\dim_k \mathfrak{L}_{2n-1}(i;a^{2n+2})\geq  \binom{2n-1+i}{2n-1} - (2n+2)\binom{2n-2+a}{2n-1}.
\end{align}

We now consider the following cases:

\noindent {\bf Case 1: n=4.} Using \eqref{equation4.1}, computations show that these linear systems are not empty for every 
\begin{align*}
e\geq \begin{cases}
1 & \text{if} \quad d-1\equiv 0,1,2 (\moddn 7)\\
0& \text{if} \quad d-1\equiv 3,4,5,6(\moddn 7).
\end{cases}
\end{align*}
In other words, $D\neq 0$ for any $d\geq 4.$

\noindent {\bf Case 2: n=5.} Using \eqref{equation4.1}, computations show that these linear system are not empty, provided 
\begin{align*}
e\geq \begin{cases}
1 & \text{if} \quad d-1\equiv 0,1,2,4 (\moddn 9)\\
0& \text{if} \quad d-1\equiv 3,5,6,7,8 (\moddn 9).
\end{cases}
\end{align*}
In other words, $D\neq 0$ for all $d\geq 4$ and $d\neq 5$.

Let $\ell_{0},\ldots,\ell_{2n+1}$ be general linear forms in $k[x_0,\ldots,x_{2n-1}]$ and set
$$P_{n,s}=k[x_0,\ldots,x_{2n-1}]/(\ell_0^s,\ldots,\ell_{2n-1}^s),\quad Q_{n,s}=k[x_0,\ldots,x_{2n-1}]/(\ell_0^s,\ldots,\ell_{2n}^s)$$
and $R_{n,s}=k[x_0,\ldots,x_{2n-1}]/(\ell_0^s,\ldots,\ell_{2n+1}^s).$ The exact sequence
\begin{align*}
\xymatrix{ [Q_{n,s}]_{i-s} \ar[rr]^{\times \ell_{2n+1}^s} && [Q_{n,s}]_{i} \ar[r]& [R_{n,s}]_{i}\ar[r]&0 }
\end{align*}
deduces that 
\begin{align*}
h_{R_{n,s}}(i)&\geq h_{Q_{n,s}}(i)-h_{Q_{n,s}}(i-s) \\
&=h_{P_{n,s}}(i) -2h_{P_{n,s}}(i-s)+h_{P_{n,s}}(i-2s)
\end{align*}
where the last equality deduces from the fact that $P_{n,s}$ is a complete intersection and has the SLP (see \cite{Stanley1980} or \cite{Watanabe1987}).

Now we need to show $D\neq 0$ for $d=5.$ Indeed, in this case, one has
\begin{align*}
D=\dim_k \mathfrak{L}_9(9; 7^{12})&=\dim_k [R_{5,3}]_9\\
&\geq h_{P_{5,3}}(9) -2h_{P_{5,3}}(6)+h_{P_{5,3}}(3)\\
&=8350 - 2\times 2850 + 210\\
&=2860
\end{align*}
which shows $D\neq 0$ for $d=5.$

\noindent {\bf Case 3: n=6.} Using \eqref{equation4.1}, computations show that these linear system are not empty for any $d\geq 6$ and $d\neq 7,9.$ We need to show  $D\neq 0$ for $d=4,5,7,9.$  If $d=4$, then $D\neq 0$, by Proposition~\ref{Corollary3.2}. With the notations as in the case 2, one has
\begin{align*}
D=\begin{cases}
\dim_k \mathfrak{L}_{11}(11; 9^{14})\;\;=\dim [R_{6,3}]_{11}& \text{if}\quad d=5\\
\dim_k \mathfrak{L}_{11}(10; 8^{14})\;\;=\dim [R_{6,3}]_{10} & \text{if}\quad d=7\\
\dim_k \mathfrak{L}_{11}(9; 7^{14})\quad=\dim [R_{6,3}]_{9} & \text{if}\quad d=9.
\end{cases}
\end{align*}	
The $h$-vector of $P_{6,3}$ is
\begin{align*}
h_{P_{6,3}} =&(1, 12, 78, 352, 1221, 3432, 8074, 16236, 28314, 43252, 58278, 69576, 73789,\\
& 69576, 58278, 43252, 28314, 16236, 8074, 3432, 1221, 352, 78, 12, 1).
\end{align*}
It is easy to see 
$$\dim [R_{6,3}]_{i} \geq h_{P_{6,3}}(i) -2h_{P_{6,3}}(i-3)+h_{P_{6,3}}(i-6) >0,$$ 
for each $i\in \{9,10,11\}.$ \\
Thus, $D>0$ for every $d\geq 4.$

\noindent {\bf Case 4: n=7.} Using \eqref{equation4.1}, computations show that these linear system are not empty for $d\geq 4$ and $d\neq 5,6,7,9,11,16$. We need to show  $D\neq 0$ for $d=5,6,7,9,11,16.$ By Proposition~\ref{Corollary3.2} we get that $D\neq 0$ for $d=6$. With the notations as in the case 2, one has
\begin{align*}
D=\begin{cases}
\dim_k \mathfrak{L}_{13}(13; 11^{16})=\dim [R_{7,3}]_{13} & \text{if}\quad d=5\\
\dim_k \mathfrak{L}_{13}(12; 10^{16})=\dim [R_{7,3}]_{12} & \text{if}\quad d=7\\
\dim_k \mathfrak{L}_{13}(11; 9^{16})\;\;=\dim [R_{7,3}]_{11} & \text{if}\quad d=9\\
\dim_k \mathfrak{L}_{13}(10; 8^{16})\;\;=\dim [R_{7,3}]_{10} & \text{if}\quad d=11\\
\dim_k \mathfrak{L}_{13}(15; 12^{16})=\dim [R_{7,4}]_{15} & \text{if}\quad d=16.
\end{cases}
\end{align*}	
As $h$-vector of $Q_{7,3}$ is
\begin{align*}
h_{Q_{7,3}}=(1, 14, 105, 545&, 2170, 6993, 18837, 43290, 85995, 148785, 224796, 295659,\\
&334425, 315420, 227475, 83097)
\end{align*}
we get $D>0$ for $d=5,7,9,11.$ Similarly, one can easy show that $D>0$ for $d=16.$ 
Thus, $D>0$ for every $d\geq 4$.

\noindent {\bf Case 5: n=8.} By Proposition~\ref{Corollary3.2}, we have $D\neq 0$ for $d=15e+2r$, $e$ and $r$ are non-negative integers such that $2\leq r\leq 8$. Using \eqref{equation4.1}, we can also show that $D\neq 0$ for $d\geq 4$ and $d\neq 5,7,9,11,13,18,20$. We now need to prove  $D\neq 0$ for $d\neq 5,7,9,11,13,18,20$. With the notations as in the case 2, one has
\begin{align*}
D=\begin{cases}
\dim_k \mathfrak{L}_{15}(15; 13^{18})=\dim [R_{8,3}]_{15} & \text{if}\quad d=5\\
\dim_k \mathfrak{L}_{15}(14; 12^{18})=\dim [R_{8,3}]_{14} & \text{if}\quad d=7\\
\dim_k \mathfrak{L}_{15}(13; 11^{18})=\dim [R_{8,3}]_{13} & \text{if}\quad d=9\\
\dim_k \mathfrak{L}_{15}(12; 10^{18})=\dim [R_{8,3}]_{12} & \text{if}\quad d=11\\
\dim_k \mathfrak{L}_{15}(11; 9^{18})\;\,=\dim [R_{8,3}]_{11} & \text{if}\quad d=13\\
\dim_k \mathfrak{L}_{15}(17; 14^{18})=\dim [R_{8,4}]_{17} & \text{if}\quad d=18\\
\dim_k \mathfrak{L}_{15}(16; 13^{18})=\dim [R_{8,4}]_{16} & \text{if}\quad d=20.
\end{cases}
\end{align*}	
As $h$-vector of $Q_{8,3}$ is
\begin{align*}
h_{Q_{8,3}}=(1, 16, 136&, 799, 3604, 13192, 40528, 106828, 245242, 495312, 885768, 1406886,\\
& 1983696,2469624, 2677704, 2448816, 1730787, 625992)
\end{align*}
we get $D>0$ for $d=5,7,9,11,13.$ Similarly, the Hilbert functions of $Q_{8,4}$ up to degree 17 are
\begin{align*}
h_{Q_{8,4}}(t)=&(1, 16, 136, 816, 3859, 15232, 51952, 156672, 424558, 1046112, 2364768,\\ 
&4937888, 9574978, 17312256, 29277264, 46411904, 69063979, 96521904)
\end{align*}
which show $D>0$ for $d=18,20.$  Thus, $D>0$ for every $d\geq 4$.

Therefore,  Claim 1 is completely proved.

Second, to prove failure of the WLP in degree $j$ it remains to show the following assertion.

\noindent \textsc{Claim 2}: $E:= \dim_k[R/I]_j -\dim_k[R/I]_{j-1} \leq 0$ for all $4\leq n\leq 8$ and $d\geq 4.$

Theorem~\ref{Theorem3.1} gives
$$E:= \dim_k[R/I]_j -\dim_k[R/I]_{j-1} = \sum_{k=0}^{n}(-1)^k \binom{2n+2}{k}\binom{2n-1+j-kd}{2n-1}.$$

We consider the following cases:

\noindent {\bf Case 1: n=4.} We consider seven cases for $d-1=7e+m, 0\leq m\leq 6$. Thank to {\tt Macaulay2} \cite{Macaulay2}, we can show that $E<0$ for any $d\geq 4.$

\noindent \underline{Subcase 1:} If $d=7e+1$, then $j=31e$ and hence
\begin{align*}
E&= \binom{31e+7}{7}-10\binom{24e+6}{7}+45\binom{17e+5}{7}-120\binom{10e+4}{7}+210\binom{3e+3}{7}\\
&= \frac{1}{7!}(-1086400574e^7 - 914853422e^6 - 328170248 e^5 - 60270140 e^4 - 5015486 e^3 \\
&\quad + 102442 e^2 +60228 e +5040) <0\quad \text{for any} \quad e\geq 1.
\end{align*}

\noindent\underline{Subcase 2:} If $d=7e+2$, then $j=31e+4$ and we have
\begin{align*}
E&= \binom{31e+11}{7}-10\binom{24e+9}{7}+45\binom{17e+7}{7}-120\binom{10e+5}{7}+210\binom{3e+3}{7}\\
&= \frac{1}{7!}(-1086400574e^7 - 1829706844e^6 - 1272885740 e^5 - 457929640 e^4 \\
&\quad - 84318206 e^3 - 5316556e^2 +535080 e +75600) <0\quad \text{for any} \quad e\geq 1.
\end{align*}

\noindent\underline{Subcase 3:} If $d=7e+3$, then $j=31e+8$. It follows that
\begin{align*}
E&= \binom{31e+15}{7}-10\binom{24e+12}{7}+45\binom{17e+9}{7}-120\binom{10e+6}{7}+210\binom{3e+3}{7}\\
&= \frac{1}{7!}(-1086400574e^7 - 2744560266e^6 - 2847411560 e^5 - 1530367860 e^4\\
&\quad - 431507006 e^3 - 50737554e^2 +1747620 e +680400) <0\quad \text{for any} \quad e\geq 1.
\end{align*}

\noindent \underline{Subcase 4:} If $d=7e+4$, then $j=31e+13$. One has
\begin{align*}
E&= \binom{31e+20}{7}-10\binom{24e+16}{7}+45\binom{17e+12}{7}-120\binom{10e+8}{7}+210\binom{3e+4}{7}\\
&= \frac{1}{7!}(-1086400574e^7 - 4059690376e^6 - 6472447730 e^5 - 5696621560 e^4\\
&\quad - 2981962256 e^3- 925181824e^2 -156720480 e -11088000) <0\quad \text{for any} \quad e\geq 0.
\end{align*}

\noindent\underline{Subcase 5:} If $d=7e+5$, then $j=31e+17$ and hence
\begin{align*}
E&= \binom{31e+24}{7}-10\binom{24e+19}{7}+45\binom{17e+14}{7}-120\binom{10e+9}{7}+210\binom{3e+4}{7}\\
&= \frac{1}{7!}(-1086400574e^7 - 4974543798 e^6 -9666743618 e^5 - 10305716610 e^4\\
&\quad  - 6484301936 e^3-2393744472 e^2 -475568352 e -38586240) <0\quad \text{for any} \quad e\geq 0.
\end{align*}

\noindent\underline{Subcase 6:} If $d=7e+6$, then $j=31e+22$ and therefore
\begin{align*}
E&= \binom{31e+29}{7}-10\binom{24e+23}{7}+45\binom{17e+17}{7}-120\binom{10e+11}{7}+210\binom{3e+5}{7}\\
&= \frac{1}{7!}(-1086400574e^7 - 6289673908 e^6 -15592053428 e^5 - 21447402760 e^4\\
&\quad  - 17672567486 e^3-8719279492 e^2 -2383703952 e -278359200) <0\; \text{for any} \; e\geq 0.
\end{align*}

\noindent\underline{Subcase 7:} If $d=7e+7$, then $j=31e+26$. It follows that
\begin{align*}
E&= \binom{31e+33}{7}-10\binom{24e+26}{7}+45\binom{17e+19}{7}-120\binom{10e+12}{7}+210\binom{3e+5}{7}\\
&= \frac{1}{7!}(-1086400574e^7 - 7204527330 e^6 -20406119384 e^5 - 31980364500 e^4\\
&\; - 29926695806 e^3-16705543050 e^2 -5144220396 e -673001280) <0\;\text{for any} \; e\geq 0.
\end{align*}

Thus $E<0$ for any $d\geq 4.$ Claim 2 is proved for $n=4.$

\noindent {\bf Case 2: n=5.} We write $d-1=9e+m, 0\leq m\leq 8$. We will prove that $E<0$ for any $d\geq 4.$ Thank to {\tt Macaulay2} \cite{Macaulay2}, a straightforward computation gives 
	
\noindent\underline{Subcase 1:} If $d=9e+1$, then $j=49e$. It follows that
\begin{align*}
E&= \sum_{k=0}^{5}(-1)^k \binom{12}{k}\binom{(49-9k)e+9-k}{9}\\
&= \frac{1}{9!}(-32649547827918e^9 -29495874488598 e^8 - 11942585863236 e^7\\
& - 2793889960092 e^6 
- 406323342558 e^5 - 35868202902 e^4 - 1535113104 e^3\\
& + 29687112 e^2 +7209216 e +362880) <0\;\text{for any} \; e\geq 1.
\end{align*}
	
Analogously we can check the another cases.

\noindent\underline{Subcase 2:} If $d=9e+2r,\ 1\leq r\leq 4$, then $j=49e+11r-6$ and
\begin{align*}
E=\sum_{k=0}^{5}(-1)^k \binom{12}{k}\binom{(49-9k)e+(11-2k)r+3}{9}.
\end{align*}
 We compute with {\tt Macaulay2} to show that if $r=1,$ then $E<0$ for any $e\geq 1$ and if $r\in \{2,3,4\}$ then $E<0$ for any $e\geq 0$.

\noindent\underline{Subcase 3:} If $d=9e+2r+1,\ 1\leq r\leq 4$, then $j=49e+11r-1$ and
\begin{align*}
E=\sum_{k=0}^{5}(-1)^k \binom{12}{k}\binom{(49-9k)e+(11-2k)r-k+8}{9}.
\end{align*}
Similarly, we can show that if $r=1,$ then $E<0$ for any $e\geq 1$ and if $r\in \{2,3,4\}$ then $E<0$ for any $e\geq 0$.

It follows that  $E<0$ for any $d\geq 4.$ Claim 2 is proved for $n=5.$

\noindent {\bf Case 3: n=6.} Write $d-1=11e+m, 0\leq m\leq 10.$  Thank to {\tt Macaulay2} \cite{Macaulay2}, we will show that $E<0$ for any $d\geq 2.$

\noindent\underline{Subcase 1:} If $d=11e+1$, then $j=71e$ and
\begin{align*}
E&=\sum_{k=0}^{6}(-1)^k \binom{14}{k}\binom{(71-11k)e+11-k}{11}= \frac{1}{11!}(-2310696921327619572 e^{11} \\
&- 2159206229822458212 e^{10} - 925836626096405100 e^9 -  238845827273630940 e^8\\
& - 40895244843536556 e^7 - 4822097086873836 e^6 - 390251062386900 e^5 \\
&-  20387890763460 e^4 - 526999267872 e^3 + 8455070448 e^2\\
& + 1189900800 e + 39916800) <0\;\text{for any} \; e\geq 1.
\end{align*}

Analogously we can check the another cases.

\noindent\underline{Subcase 2:} If  $d=11e+2r,\ 1\leq r\leq 5$, then $j=71e+13r-7$. For each $1\leq r\leq 5$, computations with {\tt Macaulay2} show that
\begin{align*}
E=\sum_{k=0}^{6}(-1)^k \binom{14}{k}\binom{(71-11k)e+(13-2k)r+4}{11}<0\quad \text{for any} \; e\geq 0.
\end{align*}

\noindent\underline{Subcase 3:} If  $d=11e+2r+1,\ 1\leq r\leq 5$, then $j=71e+13r-1$. For each $1\leq r\leq 5$, computations with {\tt Macaulay2} show that
\begin{align*}
E=\sum_{k=0}^{6}(-1)^k \binom{14}{k}\binom{(71-11k)e+(13-2k)r-k+10}{11}<0\quad \text{for any} \; e\geq 0.
\end{align*}

It follows that $E<0$ for any $d\geq 2.$ Claim 2 is proved for $n=6.$

\noindent {\bf Case 4: n=7.} We write $d-1=13e+m, 0\leq m\leq 12.$  Thank to {\tt Macaulay2} \cite{Macaulay2}, we will show that $E<0$ for any $d\geq 2.$

\noindent\underline{Subcase 1:} If $d=13e+1$, then $j=97e$ and hence
{ \begin{align*}
E&= \sum_{k=0}^{7}(-1)^k \binom{16}{k}\binom{(97-13k)e+13-k}{13}\\
&= \frac{1}{13!}(- 334688414610649890510291 e^{13} - 318779633066827110608001 e^{12} \\
&- 141329943714960759520905 e^{11} - 38495945182007845679433 e^{10} \\
&- 7165747937184385180203 e^9 -  958746457198704734703 e^8\\
&- 94270438988259145755 e^7 - 6819523889292264579 e^6 \\
&- 354359614333473606 e^5 -  12260161531299396 e^4 - 210757791455640 e^3 \\
&+ 2848164688512 e^2 + 260089315200 e + 6227020800 ) <0\;\text{for any} \; e\geq 1.
\end{align*}}

Analogously we can check the another cases.

\noindent\underline{Subcase 2:} If  $d=13e+2r,\ 1\leq r\leq 6$, then $j=97e+15r-8$. For each $1\leq r\leq 6$, computations with {\tt Macaulay2} show that
\begin{align*}
E= \sum_{k=0}^{7}(-1)^k \binom{16}{k}\binom{(97-13k)e+(15-2k)r+5}{13} <0\;\text{for any} \; e\geq 0.
\end{align*}

\noindent\underline{Subcase 3:} If  $d=13e+2r+1,\ 1\leq r\leq 6$, then $j=97e+15r-1$. For each $1\leq r\leq 6$, computations with {\tt Macaulay2} show that
\begin{align*}
E= \sum_{k=0}^{7}(-1)^k \binom{16}{k}\binom{(97-13k)e+(15-2k)r-k+12}{13} <0\;\text{for any} \; e\geq 0.
\end{align*}
It follows that $E<0$ for all $d\geq 2$ as desired.

\noindent {\bf Case 5: n=8.} Write $d-1=15e+m, 0\leq m\leq 14.$  Thank to {\tt Macaulay2} \cite{Macaulay2}, we will show that $E<0$ for any $d\geq 2.$

\noindent\underline{Subcase 1:} If $d=15e+1$, then $j=127e$ and hence
{ \begin{align*}
	E&= \sum_{k=0}^{8}(-1)^k \binom{18}{k}\binom{(127-15k)e+15-k}{15}\\
	&= \frac{1}{15!}(- 89416180762084130597433031670 e^{15}  - 86189264600012090365415830692 e^{14} \\
	&-  39053028448507299529147674830e^{13} - 11015489694695869227569915190 e^{12} \\
	&- 2159771261721698841245859830 e^{11} - 311249224672723122942089934 e^{10} \\
	&- 33979437594069966555524110  e^9 -  2851416027092144483798970 e^8\\
	&-  184330466550469812352780 e^7 - 9076381685750429456406 e^6 \\
	&- 329488524726287066140 e^5 -  8109990225233736840 e^4 - 98835121150056720 e^3 \\
	&+ 1138217439820032 e^2 + 71328551374080 e + 1307674368000 ) <0\;\text{for any} \; e\geq 1.
	\end{align*}}

Analogously we can check the another cases.

\noindent\underline{Subcase 2:} If  $d=15e+2r,\ 1\leq r\leq 7$, then $j=127e+17r-9$. For each $1\leq r\leq 7$, computations with {\tt Macaulay2} show that
\begin{align*}
E= \sum_{k=0}^{8}(-1)^k \binom{18}{k}\binom{(127-15k)e+(17-2k)r+6}{15} <0\;\text{for any} \; e\geq 0.
\end{align*}

\noindent\underline{Subcase 3:} If  $d=15e+2r+1,\ 1\leq r\leq 7$, then $j=127e+17r-1$. For each $1\leq r\leq 7$, computations with {\tt Macaulay2} show that
\begin{align*}
E= \sum_{k=0}^{8}(-1)^k \binom{18}{k}\binom{(127-15k)e+(17-2k)r-k+14}{15} <0\;\text{for any} \; e\geq 0.
\end{align*}
It follows that $E<0$ for all $d\geq 2$ and $n=8$.

Thus Claim 2 is completely proved.

Finally, Theorem~\ref{Theorem4.1} follows from the above two claims.
\end{proof}

\begin{Rem}\label{Remark4.2}\quad
\begin{enumerate}
\item The first author has shown that an artinian ideal $I=(L_0^2,\ldots,L_{2n+1}^2)\subset R$ generated by the quadratic powers of general linear forms fails to have the WLP \cite{MiroRoig2016}. Therefore, Theorem~\ref{Theorem4.1} answers partially Conjecture~\ref{Conj1.1} for $4\leq n\leq 8$, missing only the case $d=3.$
\item  Theorem~\ref{Theorem4.1} together with Corollaries~\ref{Corollary3.3}--\ref{Corollary3.10} says that $R/I$ fails to have the WLP for all $d=2r, 2\leq r\leq 8$ and $4\leq n\leq 2r(r+2)-1.$
\end{enumerate}
\end{Rem}

\section*{Acknowledgments}
Computations using the algebra software \texttt{Macaulay2} \cite{Macaulay2} were essential to get the ideas behind some of the proofs. The authors thank the referee for a careful reading and useful comments that improved the presentation of the article. The first author was partially supported by the grant MTM2016-78623-P. The second author was partially supported by the project ``\`Algebra i Geometria Algebraica" under grant number 2017SGR00585 and  by Vietnam National Foundation for Science and Technology Development (NAFOSTED) under grant number 101.04-2019.07.

\bibliographystyle{plain} 
\bibliography{bibliothese_WLP} 

\begin{thebibliography}{10}

\bibitem{AA18}
Charles Almeida and Aline~V. Andrade.
\newblock Lefschetz property and powers of linear forms in {$\Bbb K[x,y,z]$}.
\newblock {\em Forum Math.}, 30(4):857--865, 2018.

\bibitem{AFA2010}
Federico Ardila and Alexander Postnikov.
\newblock Combinatorics and geometry of power ideals.
\newblock {\em Trans. Amer. Math. Soc.}, 362(8):4357--4384, 2010.

\bibitem{GIV2014}
Roberta Di~Gennaro, Giovanna Ilardi, and Jean Vall\`es.
\newblock Singular hypersurfaces characterizing the {L}efschetz properties.
\newblock {\em J. Lond. Math. Soc. (2)}, 89(1):194--212, 2014.

\bibitem{GIV19}
Roberta Di~Gennaro, Giovanna Ilardi, and Jean Vall{\`e}s.
\newblock {Comments on the WLP in degree $3$ for $8$ cubes in ${\mathbb
  {P}}^6$}.
\newblock Preprint: hal-02321118, October 2019.

\bibitem{Dumnicki2009}
Marcin Dumnicki.
\newblock An algorithm to bound the regularity and nonemptiness of linear
  systems in {$\Bbb P^n$}.
\newblock {\em J. Symbolic Comput.}, 44(10):1448--1462, 2009.

\bibitem{EI1995}
Jacques Emsalem and Anthony Iarrobino.
\newblock Inverse system of a symbolic power. {I}.
\newblock {\em J. Algebra}, 174(3):1080--1090, 1995.

\bibitem{Macaulay2}
Daniel~R. Grayson and Michael~E. Stillman.
\newblock Macaulay2, a software system for research in algebraic geometry.
\newblock Available at http://www.math.uiuc.edu/Macaulay2/.

\bibitem{HSS2011}
Brian Harbourne, Hal Schenck, and Alexandra Seceleanu.
\newblock Inverse systems, {G}elfand-{T}setlin patterns and the weak
  {L}efschetz property.
\newblock {\em J. Lond. Math. Soc. (2)}, 84(3):712--730, 2011.

\bibitem{GV19}
Giovanna {Ilardi} and Jean {Vall{\`e}s}.
\newblock {Eight cubes of linear forms in $\mathbb{P}^6$}.
\newblock {\em arXiv e-prints: 1910.04035}, Oct 2019.

\bibitem{MMR17}
Juan~C. Migliore and Rosa~M. Mir\'{o}-Roig.
\newblock On the strong {L}efschetz problem for uniform powers of general
  linear forms in {$k[x,y,z]$}.
\newblock {\em Proc. Amer. Math. Soc.}, 146(2):507--523, 2018.

\bibitem{MMN2012}
Juan~C. Migliore, Rosa~M. Mir\'{o}-Roig, and Uwe Nagel.
\newblock On the weak {L}efschetz property for powers of linear forms.
\newblock {\em Algebra Number Theory}, 6(3):487--526, 2012.

\bibitem{MiroRoig2016}
Rosa~M. Mir\'{o}-Roig.
\newblock Harbourne, {S}chenck and {S}eceleanu's conjecture.
\newblock {\em J. Algebra}, 462:54--66, 2016.

\bibitem{NT2018}
Uwe Nagel and Bill Trok.
\newblock Interpolation and the weak {L}efschetz property.
\newblock {\em Trans. Amer. Math. Soc.}, 372(12):8849--8870, 2019.

\bibitem{SS2010}
Hal Schenck and Alexandra Seceleanu.
\newblock The weak {L}efschetz property and powers of linear forms in {$\Bbb
  K[x,y,z]$}.
\newblock {\em Proc. Amer. Math. Soc.}, 138(7):2335--2339, 2010.

\bibitem{Stanley1980}
Richard~P. Stanley.
\newblock Weyl groups, the hard {L}efschetz theorem, and the {S}perner
  property.
\newblock {\em SIAM J. Algebraic Discrete Methods}, 1(2):168--184, 1980.

\bibitem{SX2010}
Bernd Sturmfels and Zhiqiang Xu.
\newblock Sagbi bases of {C}ox-{N}agata rings.
\newblock {\em J. Eur. Math. Soc. (JEMS)}, 12(2):429--459, 2010.

\bibitem{Watanabe1987}
Junzo Watanabe.
\newblock The {D}ilworth number of {A}rtinian rings and finite posets with rank
  function.
\newblock In {\em Commutative algebra and combinatorics ({K}yoto, 1985)},
  volume~11 of {\em Adv. Stud. Pure Math.}, pages 303--312. North-Holland,
  Amsterdam, 1987.

\end{thebibliography}


\end{document}